\documentclass[a4paper,11pt,reqno]{amsart}

\usepackage{amssymb,amsmath,amsfonts} 
\usepackage{graphicx}
\usepackage{color}
\usepackage{amsthm}
\usepackage{enumerate}
\usepackage{geometry}
\numberwithin{equation}{section}    

\theoremstyle{plain}
\newtheorem{theo}{Theorem}[section]
\newtheorem{lem}[theo]{Lemma}
\newtheorem{prop}[theo]{Proposition}

\theoremstyle{definition}
\newtheorem{mydef}[theo]{Definition}
\newtheorem{rem}[theo]{Remark}
\newtheorem{bsp}[theo]{Example}

\title[A note on uniform convergence of Wiener-Wintner ergodic averages]{A note on uniform convergence of Wiener-Wintner ergodic averages}
\author[M. B\"ahring]{Markus B\"ahring}
\address{Max Planck Institut f\"ur Mathematik in den Naturwissenschaften, Leipzig, Germany}
\email{\texttt{markus.baehring@mis.mpg.de}}



\begin{document}

\begin{abstract}
We show uniform convergence of Wiener-Wintner ergodic averages for ergodic actions of (not necessarily countable) locally compact, second countable, abelian (LCA) groups. As a by-product, we obtain a finitary version of the van der Corput inequality for such groups.
\end{abstract}

\maketitle

\section{Introduction}
Let $(X,\mu)$ be a probability space and let $T: X \rightarrow X$ be a measure preserving transformation. The famous Wiener-Wintner theorem (see \cite{ww}) states that for every $f \in L_1(X,\mu)$ there exists a full measure subset $X' \subset X$ such that the averages
\begin{align}
\frac{1}{N} \sum\limits_{n=1}^N f(T^nx) \lambda^n
\end{align}
converge as $N \rightarrow \infty$ for every $x \in X'$ and every $\lambda \in \mathbb{T}$, where $\mathbb{T}$ denotes the unit circle. There are several proofs of this result, see for instance \cite{assani}, \cite{bellow} and \cite{furstenberg}. An important tool for one of the proofs is the decomposition 
\begin{align}
L_2(X,\mu)= H_{\text{kr}} \oplus H_{\text{wmix}},
\end{align}
where
\begin{align*}
H_\text{kr} = \overline{\text{span}}\left\{ f \in L_2(X,\mu) : f \circ T = \lambda f \text{ for some } \lambda \in \mathbb{T} \right\}
\end{align*}
is the \textbf{Kronecker factor} and 
\begin{align*}
H_\text{wmix} = \left\{ f \in L_2(X,\mu) : \frac{1}{N} \sum\limits_{n=1}^N \left| \int_X f(T^nx) \overline{f(x)} d \mu(x) \right| \overset{N \rightarrow \infty}\longrightarrow 0 \right\}
\end{align*}
is the \textbf{weakly mixing part}. The convergence of the averages in (1.1) on the Kronecker factor is straightforward and for the convergence on the weakly mixing part the van der Corput lemma (see \cite[Lemma 9.28]{eisner}) is used.\\
\indent Over the years this result has been improved and generalized in several ways. For instance Bourgain observed that the convergence of the averages in (1.1) for $f \in H_\text{wmix}$ is uniform in $\lambda$ (see \cite{bourgain}), cf. Assani \cite{assani}. Lesigne showed, that it is possible to replace $(\lambda^n)$ by polynomial sequences of the form $(e^{i P(n)})$, $P \in \mathbb{R}[X]$ (see \cite{lesign1}, \cite{lesign2}). A joint extension of this result and Bourgain's observation has been obtained by Frantzikinakis \cite{franz}.  Moreover, Host and Kra generalized the Wiener-Wintner theorem to the class of nilsequences \cite[Theorem 2.22]{host}. Eisner and Zorin-Kranich proved the corresponding uniform version (see \cite{eiskra}).\\
\indent A topological Wiener-Wintner Theorem is due to Robinson  \cite{robinson}, while Assani proved that for a uniquely ergodic system the convergence is uniform in $\lambda$ and $x \in X$ \cite[Theorem 2.10]{assani}. Recently Fan proved a topological version of Lesigne's Wiener-Wintner theorem for polynomial sequences \cite{fan}.\\
\indent Furthermore the Wiener-Wintner ergodic theorem has been transferred to actions of LCA-groups or more general amenable groups on probability spaces (see for example Zorin-Kranich \cite{pavel} as well as Schreiber \cite{schreiber} and Bartoszek, {\'S}piewak \cite{bartos} for the topological case). In the case of abelian groups the corresponding averages look as follows
\begin{align}
\frac{1}{m_G(F_n)} \int_{F_n} \xi(g) f(gx) dm_G(g),
\end{align}
where $(F_n)$ is a (tempered strong) F{\o}lner-sequence in $G$, $m_G$ is the Haar-measure on $G$, $\xi$ is a character of $G$ and $gx$ describes the action of an element $g \in G$ on $x \in X$.\\
\indent A uniform version of the Wiener-Wintner Theorem for abelian groups was proven by Lenz \cite[Corollary 1]{lenz} who showed that the convergence of (1.3) to zero for $f$ orthogonal to the corresponding Kronecker factor is uniform in $\xi$ (the set of characters) at least for uniquely ergodic actions of discrete LCA-groups. In this note we drop the assumptions of both unique ergodicity and discreteness and show the following generalization of Bourgain's observation and Assani's result mentioned above.
\begin{theo}[Uniform convergence of Wiener-Wintner ergodic averages]
Let $G$ be an LCA-group with Haar-measure $m_G$ acting continuously on a compact space $X$ and denote by $\hat{G}$ the dual group of $G$. Further let $\mu$ be an ergodic probability measure on $X$, $f \in L_2(X,\mu)$ with $f \in H_\text{wmix}$ and $(F_n)$ a tempered strong F{\o}lner-sequence in $G$. Then for almost every $x \in X$ 
\begin{align*}
\lim\limits_{n \rightarrow \infty} \sup\limits_{\xi \in \hat{G}} \left| \frac{1}{m_G(F_n)} \int_{F_n} \xi(g) f(gx) dm_G(g) \right| =0.
\end{align*}
Moreover, if the action of $G$ on $X$ is uniquely ergodic, then for all $f \in C(X) \cap H_\text{wmix}$
\begin{align*}
\lim\limits_{n \rightarrow \infty} \sup\limits_{\xi \in \hat{G}} \left\| \frac{1}{m_G(F_n)} \int_{F_n} \xi(g) f(gx) dm_G(g) \right\|_\infty =0.
\end{align*}
\end{theo}
Note that the first part of the theorem is a generalization of a result by Lenz (see \cite[Corollary 1]{lenz}), who considered only uniquely ergodic actions of discrete LCA-groups. The second part of the theorem is very similiar but still different to results by Lenz (see \cite[Theorem 2]{lenz}) and Schreiber (see \cite[Corollary 1.13]{schreiber}). We want to point out that in the latter works the F{\o}lner-sequence in question is not assumed to be tempered. Schreiber even considered general (not necessarily strong) F{\o}lner-sequences. The main difference is that in both results the supremum is taken only over compact subsets $\Lambda \subset \hat{G}$.
\\ \\
\indent As an important tool for the proof we present a finitary version of van der Corput's inequality for complex-valued functions on an LCA-group, see Lemma 4.1. Moreover, to overcome the main technical difficulty for uncountable groups, we show the existence of a full measure subset, such that uncountably many  F{\o}lner-averages $\frac{1}{m_G(F_n)} \int_{F_n} f(hgx) \overline{f(gx)} dm_G(g)$ converge simultaneously for all $h \in G$ on this set, see Lemma 3.4.\\ \\
\indent Additionally in Section 3 we give a direct proof of the corresponding non-uniform Wiener-Wintner Theorem due to Zorin-Kranich \cite[Corollary 4.1]{pavel} via a decomposition in the spirit of (1.2), see Theorem 3.1 below.\\
\section{Tools}
\subsection{F{\o}lner-sequences}
We recall one of many equivalent definitions of amenability. An lcsc-group $G$ with left Haar-measure $m_G$ is called \textbf{amenable} if it admits a so-called (left) F{\o}lner-sequence $(F_n)$ in the sense of the following definition.
\newpage
\begin{mydef} 
Let $G$ be an lcsc-group with left Haar-measure $m_G$.
\begin{itemize}
\item[(i)] A sequence $(F_n)$ of nonempty compact sets is called a \textbf{(left) F{\o}lner-sequence} if
\begin{align}
\frac{m_G(K F_n \Delta F_n)}{m_G(F_n)} \overset{n \rightarrow \infty}\longrightarrow 0
\end{align}
holds for every compact set $K \subset G$.
\item[(ii)] A sequence $(F_n)$ of nonempty compact sets is called a \textbf{strong (left) F{\o}lner-sequence} if
\begin{align*}
\frac{m_G(\partial_K(F_n))}{m_G(F_n)} \overset{n \rightarrow \infty}\longrightarrow 0
\end{align*}
holds for every compact set $K \subset G$, where 
\begin{align*}
\partial_K(F_n) = \left\{ g \in G \text{ } : \text{ } Kg \cap F_n \not= \emptyset \text{ and } Kg \cap (G \backslash F_n) \not= \emptyset \right\} 
\end{align*}
is called the \textbf{$K$-boundary} of the set $F_n$.
\item[(iii)] A (left) F{\o}lner-sequence $(F_n)$ is called \textbf{tempered} if there exists some $C>0$ such that for all $n \in \mathbb{N}$
\begin{align*}
m_G \left( \bigcup_{k <  n}F_k^{-1} F_n \right) \le C m_G(F_n).
\end{align*}
\end{itemize}
\end{mydef}
Compact groups, abelian groups and solvable groups are examples of amenable groups. The free group $\mathbb{F}_2$ on two generators is not amenable (see \cite[Example 1.1.5 and 1.2.11]{runde}). For more informations see for instance \cite{green}, \cite{pier} and \cite{runde}.
\begin{rem}
\begin{itemize}
\item[(i)] The convergence in (2.1) is uniform, more precisely \linebreak $\frac{m_G(k F_n \Delta F_n)}{m_G(F_n)} \overset{n \rightarrow \infty}\longrightarrow 0$ uniform for $k \in K$, for every compact $K$ (see \cite[Theorem 3]{emerson}).
\item[(ii)] Every (left) F{\o}lner-sequence admits a tempered (left) F{\o}lner-subsequence (see \cite[Prop. 1.5]{linden}), which means that in particular every amenable group has a tempered (left) F{\o}lner-sequence.
\item[(iii)] It is straightforward to show that if $G$ is an unimodular lcsc-group and $(F_n)$ a left F{\o}lner-sequence in $G$, then $( F_n^{-1} )$ is a right F{\o}lner-sequence in $G$ and vice versa.
\item[(iv)] Every strong F{\o}lner-sequence is also a (usual) F{\o}lner-sequence. For countable groups also the converse is true (see \cite[Lemma 2.8]{pog}). Moreover, every amenable lcsc-group admits a strong F{\o}lner-sequence (see \cite[Lemma 2.6]{schwarzenberger} and note that the proof does not rely on the unimodularity of $G$) and therefore, using (ii), a tempered strong F{\o}lner-sequence.
\item[(v)] For the $K$-boundary of a subset $F \subset G$ we have
\begin{align}
\partial_K(F)=K^{-1}F \cap K^{-1}(G \backslash F)
\end{align}
(see \cite[Proposition 2.2]{pog}).
\end{itemize}
\end{rem}
\noindent Later on we will need the following elementary lemma.
\begin{lem}
Let $G$ be an lcsc-group with left Haar-measure $m_G$, $V$ a compact, symmetric neighbourhood of the neutral element $e \in G$ and $F \subset G$ a compact subset. Then
\begin{align}
VF \subseteq F \cup \partial_V (F).
\end{align}
\end{lem}
\begin{proof}
We have
\begin{align*}
V F = (V F \cap F) \cup (V F \backslash F ) \subseteq F \cup (V F \backslash F).
\end{align*}
Therefore it is enough to show that
\begin{align*}
V F \backslash F \subseteq \partial_V(F).
\end{align*}
This can be verified as follows
\begin{align*}
V F \backslash F = VF \cap (G \backslash F) \overset{V \text{ unit ngbh.}}\subseteq V F \cap V (G \backslash F) \overset{V \text{ sym.}}= V^{-1} F \cap V^{-1} (G \backslash F) \overset{(2.2)}= \partial_V(F).
\end{align*}
\end{proof}
\noindent The following is well-known. We give a proof for the reader's convenience.
\begin{prop}
Let $G$ be an amenable group with left Haar-measure $m_G$, $\xi: G \longrightarrow \mathbb{T}$ a continuous homomorphism and $( F_n )$ a left F{\o}lner-sequence in $G$, then
\begin{align*}
A_n^\xi:=\frac{1}{m_G(F_n)} \int_{F_n} \xi(g) dm_G(g) \overset{n \rightarrow \infty}\longrightarrow \begin{cases}
1, \quad \text{if } \xi \equiv 1,\\
0, \quad \text{else}.
\end{cases}
\end{align*}
\begin{proof}
Let $\xi \not\equiv 1$. Then we have $|A_n^\xi| \le \frac{1}{m_G(F_n)} \int_{F_n} |\xi(g)| dm_G(g) = \frac{m_G(F_n)}{m_G(F_n)}=1$, which means, that the sequence is bounded. Therefore we can find a convergent subsequence. So let $(A_{n_k}^\xi)$ be such a sequence with $A_{n_k}^\xi \overset{k \rightarrow \infty}\longrightarrow a \in \mathbb{C}$. It follows for all $h \in G$
\begin{align*}
|\xi(h) \cdot a - a | &= \lim\limits_{k \rightarrow \infty} \frac{1}{m_G(F_{n_k})} \left| \int_{F_{n_k}} \xi(hg) dm_G(g) - \int_{F_{n_k}} \xi(g) dm_G(g) \right|\\
&= \lim\limits_{k \rightarrow \infty} \frac{1}{m_G(F_{n_k})} \left| \int_{hF_{n_k} \backslash (hF_{n_k} \cap F_{n_k})} \xi(g) dm_G(g) - \int_{F_{n_k} \backslash (hF_{n_k} \cap F_{n_k})} \xi(g) dm_G(g) \right|\\
&\le \lim\limits_{k \rightarrow \infty} \frac{m_G(hF_{n_k} \backslash (hF_{n_k} \cap F_{n_k})) + m_G(F_{n_k} \backslash (hF_{n_k} \cap F_{n_k}))}{m_G(F_{n_k})}\\
&= \lim\limits_{k \rightarrow \infty} \frac{m_G(hF_{n_k} \Delta F_{n_k})}{m_G(F_{n_k})}=0
\end{align*}
by the F{\o}lner-property  of $( F_n )$. As $\xi \not\equiv 1$ there is a $h \in G$ with $\xi(h) \not=1$ and therefore $a=0$ must hold. Now $(A_n^\xi)$ is a bounded sequence and every convergent subsequence converges to $0$. This implies that $(A_n^\xi)$ converges to $0$ itself.
\end{proof}
\end{prop}
\subsection{Actions of groups}
Suppose that $G$ is an lcsc-group and $X$ is a compact, topological space. Then $X$ becomes a measure space if we equip it with its Borel-$\sigma$-algebra $\mathfrak{B}$. By $gx$ we denote the action of an element $g \in G$ on $x \in X$. A $G$-invariant measure $\mu$ on $X$ is called \textbf{ergodic} if $\mu(gB \Delta B)=0$ for a measurable set $B \in \mathfrak{B}$ and all $g \in G$ is only possible if $\mu(B) \in \left\{ 0,1 \right\}$.\\
\indent For every action of an amenable group $G$ on a compact topological space $X$ by homeomorphisms there exists at least one $G$-invariant (and also at least one ergodic) Borel-measure $\mu$ on $X$ (see \cite[Remark I.3.4]{bekka}). If this measure is unique (and hence automatically ergodic, see \cite[Prop. 3.1]{bekka}), then we call the action of $G$ on $X$ \textbf{uniquely ergodic}.\\
\indent Every measure-preserving action of a group $G$ on a probability space $(X,\mu)$ induces a family of unitary operators $T_g: L_2(X,\mu) \longrightarrow L_2(X,\mu)$ by $T_gf(x)= f(g^{-1}x)$. In particular $T: g \mapsto T_g$ is a unitary representation of $G$ on $L_2(X,\mu)$, called \textbf{Koopman representation}. Now one defines the following subspace of $L_2(X,\mu)$
\begin{align*}
H_\text{kr} := \overline{\text{span}} \left\{ f \in L_2(X,\mu) : \text{ there exists a } \xi \in \text{Hom}(G,\mathbb{T})  \text{ such that }  T_gf = \xi(g) f \text{ for all } g \in G \right\}
\end{align*}
where $\text{Hom}(G,\mathbb{T})$ denotes the set of continuous group homomorphism from $G$ into the unit circle $\mathbb{T}$. Note that $\text{Hom}(G,\mathbb{T}) = \hat{G}$, where $\hat{G}$ denotes the dual group of $G$, in the abelian case. We obtain a decomposition $L_2(X,\mu) = H_\text{kr} \oplus H_\text{wmix}$, where $H_\text{wmix} := H_\text{kr}^\perp$. The following lemma gives a characterization of $H_\text{wmix}$ for actions of LCA-groups (see \cite[Prop 3.2 and Corollary 3.3]{duven}).
\begin{lem}[Characterization of $H_\text{wmix}$ for LCA-groups]
Let $(X,\mu)$ be a probability space, $G$ an LCA-group with Haar-measure $m_G$ acting measure preserving on $(X, \mu)$, $( F_n ), (F_n^{'})$ F{\o}lner-sequences in $G$ and $f \in L_2(X,\mu)$. Then the following assertions are equivalent.
\begin{itemize}
\item[(i)] $f \in H_\text{wmix}$.
\item[(ii)] $\displaystyle \lim\limits_{n \rightarrow \infty} \frac{1}{m_G(F_n)} \int_{F_n} | \langle h, T_gf \rangle | dm_G(g)=0$ for all $h \in L_2(X,\mu)$.
\item[(iii)] $\displaystyle \lim\limits_{n \rightarrow \infty} \frac{1}{m_G(F_n)m_G(F_n')} \int_{F_n} \int_{F_n'} | \langle h, T_{g_1} T_{g_2}f \rangle | dm_G(g_1) dm_G(g_2)=0$ for all $h \in L_2(X,\mu)$.
\end{itemize}
\end{lem}
\subsection{Generic points} Let $(X,\mu)$ be a probability space on which an amenable lcsc-group $G$ acts ergodically. We say that $x \in X$ is a \textbf{generic point} for a measurable function $f$ with $\int_X |f(y)| d\mu(y) < \infty$ (with respect to the  F{\o}lner-sequence $(F_n)$) if 
\begin{align*}
\frac{1}{m_G(F_n)} \int_{F_n} f(gx) dm_G(g) \overset{n \rightarrow \infty}\longrightarrow \int_X f(y) d\mu(y).
\end{align*}
If $(F_n)$ is a tempered F{\o}lner-sequence then for every such $f$ almost every $x \in X$ is generic by the Birkhoff ergodic theorem for amenable groups (see \cite[Theorem 3.1]{linden}).
\section{Proof of Wiener-Wintner Theorem for abelian groups}
The following Wiener-Wintner theorem for LCA-groups is in a more general setting due to Zorin-Kranich \cite[Corollary 4.1]{pavel}. 
\begin{theo}[Wiener-Wintner Theorem for actions of LCA-groups]
Let $G$ be an LCA-group with Haar-measure $m_G$, which acts continuously on a compact space $X$. Further let $\mu$ be an ergodic probability measure on $X$ and $f \in L_1(X,\mu)$. Then there exists a subset $X' \subset X$ with $\mu(X')=1$ such that for every tempered strong F{\o}lner-sequence $( F_n )$ in $G$ 
\begin{align*}
\frac{1}{m_G(F_n)} \int_{F_n} \xi(g) f(gx) dm_G(g)
\end{align*}
converges as $n \rightarrow \infty$ for all $\xi \in \hat{G}$ and for all $x \in X'$.
\end{theo}
Here we will present an alternative direct proof based on the decomposition of $L_2(X,\mu)$ discussed above. We need three technical lemmata. The first one is an infinitary version of the van der Corput lemma for groups (see \cite[Lemma 5.1]{duven}). For more versions of the van der Corput lemma for groups see for instance \cite{duven2}.
\begin{lem}
Let $G$ be an amenable group with left Haar-measure $m_G$ and $( F_n )$ a left F{\o}lner-sequence in $G$. Further let $f: G \rightarrow \mathbb{C}$ be a bounded, Borel measurable function. If for every $h \in G$
\begin{align*}
\gamma_h:= \limsup\limits_{n \rightarrow \infty} \left| \frac{1}{m_G(F_n)} \int_{F_n} f(hg) \overline{f(g)} dm_G(g) \right|
\end{align*}
exists, then the following implication holds
\begin{align*}
\lim\limits_{n \rightarrow \infty} \frac{1}{m_G(F_n)^2}\int_{F_n} \int_{F_n} \gamma_{h_1h_2^{-1}} dm_G(h_1)dm_G(h_2) =0 \quad \implies
\quad \lim\limits_{n \rightarrow \infty} \frac{1}{m_G(F_n)} \int_{F_n} f(g) dm_G(g) =0.
\end{align*}
\end{lem}
The next lemma will be crucial for the approximation argument. It is a simple modification of \cite[Lemma 21.7]{eisner}. One just needs to replace the Ces\`{a}ro-averages $N^{-1} \sum_{n=1}^N$ by F{\o}lner-averages $m(F_n)^{-1} \int_{F_n}$ and therefore we omit the proof.
\begin{lem}
Let $G$ be an amenable group with left Haar-measure $m_G$, $( F_n )$ a left F{\o}lner-sequence in $G$ and $\xi \in \text{Hom}(G,\mathbb{T})$. Let $G$ act ergodically on $(X,\mu)$ and suppose $f, f_1, f_2, \ldots$ are integrable functions on $X$ such that $\|f-f_j\|_{L_1} \overset{j \rightarrow \infty}\longrightarrow 0$. If $x$ is generic for $|f_j|$ and $|f-f_j|$ for all $j \in \mathbb{N}$ and if the limits
\begin{align*}
\lim\limits_{n \rightarrow \infty} \frac{1}{m_G(F_n)} \int_{F_n} \xi(g) f_j(gx) dm_G(g):= b_j
\end{align*}
exist for every $j \in \mathbb{N}$, then also the limit $\lim\limits_{j \rightarrow \infty} b_j =:b$ exists and 
\begin{align*}
\lim\limits_{n \rightarrow \infty} \frac{1}{m_G(F_n)} \int_{F_n} \xi(g) f(gx)dm_G(g) =b.
\end{align*}
\end{lem}
Finally the following lemma is needed to construct the full measure subset of convergence for actions of uncountable groups.
\begin{lem}
Let $G$ be an LCA-group with Haar-measure $m_G$, which acts continuously on a compact space $X$. Further let $\mu$ be an ergodic measure on $X$, $f \in L_\infty(X,\mu)$ and $(F_n) \subset G$ a tempered strong F{\o}lner-sequence in $G$. Then there exists a subset $X' \subset X$ with $\mu(X')=1$ such that for all $x \in X'$ and for all $h \in G$ 
\begin{align}
\frac{1}{m_G(F_n)} \int_{F_n} f(hgx) \overline{f(gx)} dm_G(g) \overset{n \rightarrow \infty}\longrightarrow \int_X f(hy) \overline{f(y)} d \mu(y) =  \langle T_{h^{-1}} f, f \rangle.
\end{align}
\end{lem}
Note that for a fixed $h \in G$ there exists a subset $X_h \subset X$ with $\mu(X_h)=1$ such that the convergence in (3.1) holds for this $h \in G$ and for all $x \in X_h$ by the Birkhoff ergodic theorem for amenable groups. So if $G$ is countable there is nothing to prove. But the uncountable case is a bit more involved.
\begin{proof}
\underline{Preparation:}\\
We need the following:
\begin{itemize}
\item For $f \in L_\infty(X)$ we can find a sequence $(f_k) \subset C(X)$ such that $f_k \overset{k \rightarrow \infty}\longrightarrow f$ in $L_1$ and $\| f_k \|_\infty \le \|f \|_\infty$ for all $k \in \mathbb{N}$. In particular for every $\varepsilon >0$ there exists a $k_0 \in \mathbb{N}$ such that
\begin{align*}
\|f-f_{k_0}\|_1 \le \varepsilon.
\end{align*}
\item Since $G$ is second countable we can find a countable dense subset $\left\{ h_i \right\}_{i=1}^\infty \subset G$. Fix a compact, symmetric neighbourhood $V$ of the unit $e \in G$, then for every $h \in G$, we can find a $h_{i_0}$ such that $h \in h_{i_0}V$. It follows that
\begin{align}
hF_n \subset h_{i_0} V F_n.
\end{align}
\item Let $X'$ be the full measure subset of $X$ which consists of the generic points of $|f-f_k|$ for every $k \in \mathbb{N}$ intersected with the set of generic points of $T_{h_i^{-1}} f_k \cdot \overline{f_k}$ for all $i \in \mathbb{N}$ and all $k \in \mathbb{N}$ and the set of generic points of $T_{h_i^{-1}} |f-f_{k}|$ for all $i \in \mathbb{N}$ and all $k \in \mathbb{N}$.
\end{itemize}
\underline{Estimation:}\\ 
Let $\varepsilon >0$ be arbitrary and $x \in X'$. \\ \\
1. We first prove that for $n$ large enough
\begin{align*}
\left| \frac{1}{m_G(F_n)} \int_{F_n} f(hgx) \overline{f(gx)} dm_G(g) - \frac{1}{m_G(F_n)} \int_{F_n} f(hgx) \overline{f_{k_0}(gx)} dm_G(g) \right| \le 2 \varepsilon \| f \|_\infty.
\end{align*}
This can be verified as follows
\begin{align*}
\bigg| \frac{1}{m_G(F_n)} \int_{F_n} &f(hgx) \overline{f(gx)} dm_G(g) - \frac{1}{m_G(F_n)} \int_{F_n} f(hgx) \overline{f_{k_0}(gx)}  dm_G(g) \bigg|\\ 
&\le \frac{1}{m_G(F_n)} \int_{F_n}  \underbrace{|f(hgx)|}_{\le \|f\|_\infty} |\overline{f(gx)} - \overline{f_{k_0}(gx)} | dm_G(g)\\
&\le \frac{1}{m_G(F_n)} \int_{F_n} 
|\overline{f(gx)} - \overline{f_{k_0}(gx)} | dm_G(g) \cdot \|f\|_\infty\\
&\overset{\text{Birkhoff}}\longrightarrow \int_X |\overline{f}(y) - \overline{f_{k_0}(y)} | d\mu(y) \cdot \| f \|_\infty = \| f -f_{k_0} \|_1 \| f \|_\infty\\
&\le \varepsilon \| f \|_\infty.
\end{align*}
Now the claim follows by choosing $n$ large enough.\\ \\
2. Now we prove that for $n$ large enough
\begin{align*}
\left| \frac{1}{m_G(F_n)} \int_{F_n} f(hgx) \overline{f_{k_0}(gx)} dm_G(g) - \frac{1}{m_G(F_n)} \int_{F_n} f_{k_0}(hgx) \overline{f_{k_0}(gx)}dm_G(g) \right| \le 2 \varepsilon \| f \|_\infty.
\end{align*}
This can be verified as follows
\begin{align*}
\bigg| \frac{1}{m_G(F_n)} \int_{F_n}  &f(hgx) \overline{f_{k_0}(gx)} dm_G(g) - \frac{1}{m_G(F_n)} \int_{F_n}  f_{k_0}(hgx) \overline{f_{k_0}(gx)} dm_G(g) \bigg|\\
&\le \frac{1}{m_G(F_n)} \int_{F_n} |f(hgx)-f_{k_0}(hgx)| \underbrace{|\overline{f_{k_0}(gx)|}}_{\le \|f_{k_0}\|_\infty \le  \|f\|_\infty} dm_G(g) \\
&\le \frac{1}{m_G(F_n)} \int_{h F_n} |f(gx)-f_{k_0}(gx)| dm_G(g) \cdot \|f\|_\infty\\
&\overset{(3.2)}\le \frac{1}{m_G(F_n)} \int_{h_{i_0} V F_n} |f(gx)-f_{k_0}(gx)| dm_G(g) \cdot \|f\|_\infty\\
&= \frac{1}{m_G(F_n)} \int_{V F_n} |f(h_{i_0}gx)-f_{k_0}(h_{i_0}gx)| dm_G(g) \cdot \|f\|_\infty\\
&\overset{(2.3)}\le \frac{1}{m_G(F_n)} \bigg( \int_{F_n} |f(h_{i_0}gx)-f_{k_0}(h_{i_0}gx)|dm_G(g)\\
&\quad \quad \quad \quad \quad \quad + \int_{\partial_V(F_n)} |f(h_{i_0}gx)-f_{k_0}(h_{i_0}gx)| dm_G(g) \bigg) \|f \|_\infty\\
&\le  \left(\frac{1}{m_G(F_n)} \int_{F_n} |f(gh_{i_0}x)-f_{k_0}(gh_{i_0}x)|dm_G(g) + \frac{m_G(\partial_V(F_n))}{m_G(F_n)} 2 \| f \|_\infty \right) \|f \|_\infty\\
&\overset{\text{Lindenstrauss} \atop \text{and strong F{\o}lner}}\longrightarrow \int_X |f(h_{i_0}y)-f_{k_0}(h_{i_0}y)| d\mu(y) \cdot \| f \|_\infty = \| f-f_{k_0}\|_1 \| f \|_\infty \le \varepsilon \| f \|_\infty.
\end{align*}
Now the claim follows by choosing $n$ large enough.\\ \\
3. Next we prove that for $h_{i}$ close to $h$
\begin{align*}
\left| \frac{1}{m_G(F_n)} \int_{F_n}  f_{k_0}(hgx) \overline{f_{k_0}(gx)} dm_G(g) -  \frac{1}{m_G(F_n)} \int_{F_n}  f_{k_0}(h_igx) \overline{f_{k_0}(gx)} dm_G(g) \right| \le \varepsilon \|f \|_\infty.
\end{align*}
This can be verified as follows
\begin{align*}
\bigg| \frac{1}{m_G(F_n)} \int_{F_n}  &f_{k_0}(hgx) \overline{f_{k_0}(gx)} dm_G(g) -  \frac{1}{m_G(F_n)} \int_{F_n}  f_{k_0}(h_igx) \overline{f_{k_0}(gx)} dm_G(g)  \bigg|\\
&\le \frac{1}{m_G(F_n)} \int_{F_n} \underbrace{|\overline{f_{k_0}(gx)}|}_{\le \| f_{k_0}\|_\infty \le \| f \|_\infty} |f_{k_0}(hgx) -  f_{k_0}(h_igx) | dm_G(g).\\
&\le \frac{1}{m_G(F_n)} \int_{F_n} |f_{k_0}(hgx) -  f_{k_0}(h_igx) | dm_G(g) \cdot \|f \|_\infty.
\end{align*}
Since the action of $G$ on the compact space $X$ is continuous and also $f_{k_0}$ is continuous, the difference $|f_{k_0}(hgx) -  f_{k_0}(h_igx) |$ is smaller than $\varepsilon$ if $h_i$ is close enough to $h$ for all $x \in X$. Now the claim follows.\\ \\
\noindent 4. The estimate
\begin{align*}
\left| \frac{1}{m_G(F_n)} \int_{F_n} f_{k_0}(h_igx) \overline{f_{k_0}(gx)} - \int_X f_{k_0}(h_iy) \cdot \overline{f_{k_0}(y)} d\mu(y) \right| \le \varepsilon
\end{align*}
for $n$ large enough is an immediate consequence of the Birkhoff ergodic theorem for amenable groups and the construction of the set $X' \ni x$.\\ \\
5. We continue by proving
\begin{align*}
\left| \int_X f_{k_0}(h_iy) \cdot \overline{f_{k_0}(y)} d\mu(y) - \int_X f_{k_0}(hy) \cdot \overline{f_{k_0}(y)} d\mu(y) \right| \le \varepsilon \| f \|_\infty
\end{align*}
for $h_i$ close to $h$. This can be verified as follows
\begin{align*}
&\left| \int_X f_{k_0}(h_iy) \cdot \overline{f_{k_0}(y)} d\mu(y) - \int_X f_{k_0}(hy) \cdot \overline{f_{k_0}(y)} d\mu(y) \right| \\
&\quad \quad \quad \quad \quad \quad \le \int_X \underbrace{|\overline{f_{k_0}}(y)|}_{\le \| f_{k_0} \|_\infty \le \| f \|_\infty} | f_{k_0}(h_iy) - f_{k_0} (hy)| d\mu(y).
\end{align*}
Again the claim follows because the difference $| f_{k_0}(h_iy) - f_{k_0} (hy)|$ is smaller than $\varepsilon$ for $h_i$ close enough to $h$ for all $y \in X$.\\ \\
\noindent 6. Now we prove that
\begin{align*}
\left| \int_X f_{k_0}(hy) \cdot \overline{f_{k_0}(y)} d\mu(y) - \int_X f(hy) \cdot \overline{f_{k_0}(y)} d\mu(y) \right| \le \varepsilon \| f \|_\infty.
\end{align*}
This can be verfied as follows
\begin{align*}
\bigg| \int_X f_{k_0}(hy) &\cdot \overline{f_{k_0}(y)} d\mu(y) - \int_X f(hy) \cdot \overline{f_{k_0}(y)} d\mu(y) \bigg| \le \int_X \underbrace{|\overline{f_{k_0}(y)}|}_{\le \|f_{k_0}\|_\infty \le \|f \|_\infty} |f_{k_0}(hy)-f(hy)| d\mu(y)\\
&\le \int_X |f_{k_0}(hy)-f(hy)| d\mu(y) \cdot \| f \|_\infty  \overset{G \curvearrowright X \atop \text{measure pres.}}= \int_X |f_{k_0}(y)-f(y)|d\mu(y) \cdot \| f \|_\infty\\
&= \| f_{k_0}-f \|_1 \| f \|_\infty \le \varepsilon \| f \|_\infty.
\end{align*}
\noindent 7. Analogously
\begin{align*}
\left| \int_X f(hy)  \cdot \overline{f_{k_0}}(y) d\mu(y) - \int_X f(hy) \cdot \overline{f(y)} d\mu(y) \right| \le \varepsilon \| f \|_\infty.
\end{align*}
\noindent \underline{Final Conclusion:}\\
Putting everything togehter yields for $n$ largh enough
\begin{align*}
\bigg| \frac{1}{m_G(F_n)} \int_{F_n} &f(hgx) \overline{f(gx)} dm_G(g) - \int_X f(hy) \cdot \overline{f(y)} d \mu(y) \bigg|\\
 &\le \varepsilon ( 2 \|f \|_\infty + 2 \|f \|_\infty  + \|f \|_\infty +1 + \|f \|_\infty +\|f \|_\infty + \|f \|_\infty )\\
 &= \varepsilon (1 + 8 \| f \|_\infty).
\end{align*}
Now the claim follows.
\end{proof}
\begin{proof}[Proof of Theorem 3.1]
First let $f \in L_\infty(X,\mu)$. Then we get $f=f_1 + f_2$ with $f_1 \in H_{\text{kr}}$ and $f_2 \in H_{\text{wmix}}$. Since $H_{\text{kr}}$ comes from a factor, which can be proven as in the classical case (see \cite{eisner}), we have $f_1 \in L_\infty(X)$ and hence also $f_2 \in L_\infty(X)$. If $f$ is an eigenfunction to the character $\eta \in \hat{G}$ we have
\begin{align*}
\frac{1}{m_G(F_n)} \int_{F_n} \xi(g) f(gx) dm_G(g) &= \frac{1}{m_G(F_n)} \int_{F_n} \xi(g) \eta(g^{-1}) f(x) dm_G(g)\\
&= \frac{1}{m_G(F_n)} \int_{F_n} \xi(g) \overline{\eta(g)} dm_G(g) f(x) \quad \text{a.s}.
\end{align*}
and this expression converges by Proposition 2.4. The convergence holds also for linear combinations of eigenfunctions to characters.\\ \\
Now suppose $f$ is in the closure, in particular $\|f-f_j\|_{L_2} \overset{j \rightarrow \infty}\rightarrow 0$, where each $f_j$ is a finite linear combination of eigenfunctions to characters. We have $\|f-f_j\|_{1} \le \|f-f_j\|_{2} \overset{j \rightarrow \infty}\rightarrow 0$. Let $\tilde{X_j} \subset X$ be the subset for which $\frac{1}{m_G(F_n)} \int_{F_n} \xi(g) f_j(gx) dm_G(g)$ converges, intersected with the set of generic points of $|f_j|$ and $|f-f_j|$ and set $\tilde{X}:= \bigcap_{j \in \mathbb{N}} \tilde{X_j}$. Then $\mu(\tilde{X})=1$ and for all $x \in \tilde{X}$ we have that $\frac{1}{m_G(F_n)} \int_{F_n} \xi(g) f(gx)dm_G(g)$ converges by Lemma 3.3.\\ \\
Now suppose $f \in H_{\text{wmix}} \cap L_\infty(X,\mu)$ and define the bounded function $\tilde{f}: G \rightarrow \mathbb{C}$ by $\tilde{f}(g):= \xi(g) f(gx)=\xi(g) T_{g^{-1}}f(x)$, where $\xi \in \hat{G}$ and $x \in X$. It holds that
\begin{align*}
\tilde{f}(hg) \overline{\tilde{f}(g)} = \xi(h) T_{g^{-1}}(T_{h^{-1}}f \overline{f} ) (x).
\end{align*}
Therefore we have
\begin{align*}
\left| \frac{1}{m_G(F_n)} \int_{F_n} \tilde{f}(hg) \overline{\tilde{f}(g)} dm_G(g) \right| \overset{n \rightarrow \infty}\longrightarrow |\langle T_{h^{-1}}f, f \rangle| =: \gamma_h,
\end{align*}
for every $x$ in the full measure subset of $X$ according to Lemma 3.4. By Lemma 2.5 we get
\begin{align*}
&\frac{1}{m_G(F_n)^2} \int_{F_n} \int_{F_n} \gamma_{h_1h_2^{-1}} dm_G(h_1) dm_G(h_2) = \frac{1}{m_G(F_n)^2} \int_{F_n} \int_{F_n} |\langle T_{h_1^{-1}h_2} f,f \rangle| dm_G(h_1) dm_G(h_2)\\
&\quad \quad \quad = \frac{1}{m_G(F_n^{-1})m_G(F_n)} \int_{F_n} \int_{F_n^{-1}} |\langle f,T_{h_1h_2}f \rangle| dm_G(h_1) dm_G(h_2) \overset{n \rightarrow \infty}\longrightarrow 0.
\end{align*}
In the first step we used here that $G$ is abelian and in the second step we used that $G$ is unimodular and therefore $\left\{ F_n^{-1} \right\}$ is again a F{\o}lner-sequence by Remark 2.2(iii). By Lemma 3.2 it follows that
\begin{align*}
\frac{1}{m_G(F_n)} \int_{F_n} \xi(g) f(gx) dm_G(g) \overset{n \rightarrow \infty}\longrightarrow 0.
\end{align*}
So far we proved the theorem for every $f \in L_\infty(X,\mu)$. Now suppose $f \in L_1(X,\mu)$. Then we can find a sequence $(f_j) \subset L_\infty(X,\mu)$ such that $\|f-f_j\|_{L_1} \overset{j \rightarrow \infty}\longrightarrow 0$. For every $j \in \mathbb{N}$ let $X_j$ be the set of points for which $\frac{1}{m_G(F_n)} \int_{F_n} \xi(g) f_j(gx) dm_G(g)$ converges, interesected with the set of generic points of $|f_j|$ and $|f-f_j|$. Then $\mu(X_j)=1$ and the same holds for $X':= \bigcap_{j \in \mathbb{N}} X_j$. Using Lemma 3.3 again yields for every $x \in X'$
\begin{align*}
\lim\limits_{n \rightarrow \infty}  \frac{1}{m_G(F_n)} \int_{F_n} \xi(g) f(gx)dm_G(g)  = \lim\limits_{j \rightarrow \infty}  \lim\limits_{n \rightarrow \infty}  \frac{1}{m_G(F_n)} \int_{F_n} \xi(g) f_j(gx)dm_G(g)
\end{align*}
and the limit exists.
\end{proof}
\section{Proof of Theorem 1.1}
We begin with the following version of van der Corput's inequality. The proof is a modification of \cite[Lemma 2.2]{assani}, see also \cite[Lemma I.3.1]{kuipers}.
\begin{lem}[Finitary version of van der Corput inequality]
Let $G$ be an amenable group with left Haar-measure $m_G$, $(F_n )$ a left F{\o}lner-sequence in $G$ and $f:G \rightarrow \mathbb{C}$ a bounded, measurable function. Then for every $n,H \in \mathbb{N}$ the following inequality holds
\begin{align*}
\left| \frac{1}{m_G(F_n)} \int_{F_n} f(g) dm_G(g) \right|^2 &\le \frac{1}{m_G(F_H)^2} \int_{F_H} \int_{F_H} \left| \frac{1}{m_G(F_n)} \int_{F_n} f(h_1h_2^{-1}g) \overline{f(g)} dm_G(g) \right| dm_G(h_1)dm_G(h_2)\\
&\quad \quad+ 3\sup\limits_{h \in F_H} \frac{m_G(hF_n \Delta F_n)}{m_G(F_n)}  \|f\|_\infty^2  + \left( \sup\limits_{h \in F_H} \frac{m_G(hF_n \Delta F_n)}{m_G(F_n)} \right)^2 \|f\|_\infty^2.
\end{align*}
\end{lem}
\begin{proof}
For every $h \in F_H$ we have
\begin{align}
&\left| \frac{1}{m_G(F_n)} \int_{F_n} f(hg) dm_G(g) \right| \nonumber \\
&\quad \quad = \bigg| \frac{1}{m_G(F_n)} \int_{hF_n} f(g) dm_G(g) + \frac{1}{m_G(F_n)} \int_{F_n} f(g) dm_G(g) - \frac{1}{m_G(F_n)} \int_{F_n} f(g) dm_G(g) \bigg| \nonumber \\
&\quad \quad \le \frac{1}{m_G(F_n)} \left| \int_{hF_n \backslash (hF_n \cap F_n)} f(g)dm_G(g) - \int_{F_n \backslash (hF_n \cap F_n)}  f(g)dm_G(g) \right|  + \left| \frac{1}{m_G(F_n)} \int_{F_n} f(g)dm_G(g) \right| \nonumber \\
&\quad \quad \le \frac{1}{m_G(F_n)} \bigg( \|f\|_\infty m_G(hF_n \backslash (hF_n \cap F_n)) + \|f\|_\infty m_G (F_n \backslash (hF_n \cap F_n)) \bigg) + \left| \frac{1}{m_G(F_n)} \int_{F_n} f(g)dm_G(g) \right| \nonumber \\
&\quad \quad = \|f\|_\infty \frac{m_G(hF_n \Delta F_n)}{m_G(F_n)} + \left| \frac{1}{m_G(F_n)} \int_{F_n} f(g)dm_G(g) \right| \nonumber \\
&\quad \quad \le \|f\|_\infty \sup\limits_{h \in F_H} \frac{m_G(hF_n \Delta F_n)}{m_G(F_n)} + \left| \frac{1}{m_G(F_n)} \int_{F_n} f(g)dm_G(g) \right|.
\end{align}
Moreover we have
\begin{align*}
&\left| \frac{1}{m_G(F_n)} \int_{F_n} f(g) dm_G(g) \right|\\
&\quad \quad \le \left| \frac{1}{m_G(F_n)} \int_{F_n} f(g)dm_G(g) - \frac{1}{m_G(F_n)} \int_{F_n} \frac{1}{m_G(F_H)} \int_{F_H} f(hg) dm_G(h)dm_G(g) \right|\\
&\quad \quad \quad \quad \quad \quad \quad \quad \quad \quad \quad \quad \quad \quad \quad  + \left| \frac{1}{m_G(F_n)} \int_{F_n} \frac{1}{m_G(F_H)} \int_{F_H} f(hg) dm_G(h)dm_G(g) \right|\\
&\quad \quad = \frac{1}{m_G(F_n)m_G(F_H)} \left| \int_{F_H} \left( \int_{F_n} f(g) dm_G(g) - \int_{F_n} f(hg) dm_G(g) \right) dm_G(h) \right|\\
&\quad \quad \quad \quad \quad \quad \quad \quad \quad \quad \quad \quad \quad \quad \quad + \left| \frac{1}{m_G(F_n)} \int_{F_n} \frac{1}{m_G(F_H)} \int_{F_H} f(hg) dm_G(h)dm_G(g) \right|\\
&\quad \quad \le \frac{1}{m_G(F_n)m_G(F_H)} \int_{F_H} \underbrace{\left| \int_{F_n} f(g) dm_G(g) - \int_{hF_n} f(g) dm_G(g) \right|}_{\le \|f\|_\infty \sup\limits_{h \in F_H} m_G(hF_n \Delta F_n)}  dm_G(h)\\
&\quad \quad \quad \quad \quad \quad \quad \quad \quad \quad \quad \quad \quad \quad \quad + \left| \frac{1}{m_G(F_n)} \int_{F_n} \frac{1}{m_G(F_H)} \int_{F_H} f(hg) dm_G(h)dm_G(g) \right|\\
&\quad \quad \le \|f\|_\infty \sup\limits_{h \in F_H} \frac{m_G(hF_n \Delta F_n)}{m_G(F_n)} + \left| \frac{1}{m_G(F_n)} \int_{F_n} \frac{1}{m_G(F_H)} \int_{F_H} f(hg) dm_G(h)dm_G(g) \right|
\end{align*}
and therefore
\begin{align}
\left| \frac{1}{m_G(F_n)} \int_{F_n} f(g) dm_G(g) \right|^2 &\le \left| \frac{1}{m_G(F_n)} \int_{F_n} \frac{1}{m_G(F_H)} \int_{F_H} f(hg) dm_G(h)dm_G(g) \right|^2 \nonumber \\
&\quad \quad + 2 \|f\|_\infty^2 \sup\limits_{h \in F_H} \frac{m_G(hF_n \Delta F_n)}{m_G(F_n)} + \left( \sup\limits_{h \in F_H} \frac{m_G(hF_n \Delta F_n)}{m_G(F_n)} \right)^2 \|f\|_\infty^2.
\end{align}
For the first term we use the Cauchy-Schwarz inequality and obtain
\begin{align*}
&\left| \frac{1}{m_G(F_n)} \int_{F_n} \frac{1}{m_G(F_H)} \int_{F_H} f(hg) dm_G(h)dm_G(g) \right|^2\\
&\quad \quad \le \frac{1}{m_G(F_n)} \int_{F_n} \left| \frac{1}{m_G(F_H)} \int_{F_H} f(hg) dm_G(h) \right|^2 dm_G(g)\\
&\quad \quad = \frac{1}{m_G(F_n)} \int_{F_n} \frac{1}{m_G(F_H)^2}  \int_{F_H} \int_{F_H} f(h_1g) \overline{f(h_2g)} dm_G(h_1)dm_G(h_2) dm_G(g)\\
&\quad \quad = \frac{1}{m_G(F_H)^2} \int_{F_H} \int_{F_H} \frac{1}{m_G(F_n)} \int_{F_n} f(h_1g) \overline{f(h_2g)} dm_G(g)  dm_G(h_1) dm_G(h_2)\\
&\quad \quad  \le \frac{1}{m_G(F_H)^2} \int_{F_H} \int_{F_H} \left| \frac{1}{m_G(F_n)} \int_{F_n} f(h_1g) \overline{f(h_2g)} dm_G(g) \right| dm_G(h_1) dm_G(h_2)\\
\end{align*}
Now we repeat the same trick as in (4.1) to obtain
\begin{align}
&\frac{1}{m_G(F_H)^2} \int_{F_H} \int_{F_H} \left| \frac{1}{m_G(F_n)} \int_{F_n} f(h_1g) \overline{f(h_2g)} dm_G(g) \right| dm_G(h_1) dm_G(h_2) \nonumber \\
&\quad \quad= \frac{1}{m_G(F_H)^2} \int_{F_H} \int_{F_H} \bigg| \frac{1}{m_G(F_n)} \int_{h_2F_n} f(h_1h_2^{-1}g) \overline{f(g)} dm_G(g) + \frac{1}{m_G(F_n)} \int_{F_n} f(h_1h_2^{-1}g) \overline{f(g)} dm_G(g) \nonumber \\
& \quad \quad \quad \quad \quad \quad\quad\quad\quad\quad \quad \quad \quad \quad \quad \quad \quad \quad \quad \quad  - \frac{1}{m_G(F_n)} \int_{F_n} f(h_1h_2^{-1}g) \overline{f(g)} dm_G(g) \bigg| dm_G(h_1) dm_G(h_2) \nonumber \\
&\quad \quad \le \frac{1}{m_G(F_H)^2} \int_{F_H} \int_{F_H} \bigg( \left| \frac{1}{m_G(F_n)} \int_{F_n} f(h_1h_2^{-1}g) \overline{f(g)} dm_G(g) \right| \nonumber \\
& \quad \quad \quad \quad \quad \quad\quad\quad\quad\quad \quad \quad \quad \quad \quad \quad \quad \quad \quad \quad + \|f\|_\infty^2 \sup\limits_{h \in F_H} \frac{m_G(hF_n \Delta F_n)}{m_G(F_n)} \bigg) dm_G(h_1) dm_G(h_2) \nonumber \\
&\quad \quad= \frac{1}{m_G(F_H)^2} \int_{F_H} \int_{F_H} \left| \frac{1}{m_G(F_n)} \int_{F_n} f(h_1h_2^{-1}g) \overline{f(g)} dm_G(g) \right| dm_G(h_1) dm_G(h_2) \nonumber \\
& \quad \quad \quad \quad \quad \quad\quad\quad\quad\quad \quad \quad \quad \quad \quad \quad \quad \quad \quad \quad + \sup\limits_{h \in F_H} \frac{m_G(h F_n \Delta F_n)}{m_G(F_n)} \|f\|_\infty^2.
\end{align}
Using (4.2) and (4.3) we finally get
\begin{align*}
&\left| \frac{1}{m_G(F_n)} \int_{F_n} f(g) dm_G(g) \right|^2\\
&\quad \le \sup\limits_{h \in F_H} \frac{m_G(h F_n \Delta F_n)}{|F_n|} \|f\|_\infty^2 + \frac{1}{m_G(F_H)^2} \int_{F_H} \int_{F_H} \left| \frac{1}{m_G(F_n)} \int_{F_n} f(h_1h_2^{-1}g) \overline{f(g)} dm_G(g) \right| dm_G(h_1) dm_G(h_2)\\
&\quad \quad \quad \quad \quad \quad \quad \quad \quad \quad \quad \quad+ 2 \|f\|_\infty^2 \sup\limits_{h \in F_H} \frac{m_G(hF_n \Delta F_n)}{m_G(F_n)} + \left( \sup\limits_{h \in F_H} \frac{m_G(hF_n \Delta F_n)}{m_G(F_n)} \right)^2 \|f\|_\infty^2\\
&\quad = \frac{1}{m_G(F_H)^2} \int_{F_H} \int_{F_H} \left| \frac{1}{m_G(F_n)} \int_{F_n} f(h_1h_2^{-1}g) \overline{f(g)} dm_G(g) \right| dm_G(h_1) dm_G(h_2)\\
&\quad \quad \quad \quad \quad \quad \quad \quad \quad \quad \quad \quad+ 3 \sup\limits_{h \in F_H} \frac{m_G(hF_n \Delta F_n)}{m(_GF_n)} \|f\|_\infty^2 + \left( \sup\limits_{h \in F_H} \frac{m_G(hF_n \Delta F_n)}{m_G(F_n)} \right)^2 \|f\|_\infty^2.
\end{align*}
\end{proof}
\noindent Now we are able to prove our main theorem.
\begin{proof}[Proof of Theorem 1.1]
First let $f$ be bounded. Without loss of generality we can assume that $f \not= 0$. We consider bounded functions $\tilde{f}: G \longrightarrow \mathbb{C}$ defined by $\tilde{f}(g)=f(gx) \xi(g)$, where $\xi \in \hat{G}$ and $x \in X$. Note that $\|\tilde{f}\|_\infty = \sup\limits_{g \in G} |f(gx) \xi(g)| = \sup\limits_{g \in G} |f(gx)| \le \|f\|_\infty$. We are going to use Lemma 4.1. We have for all $x \in X$ out of the full measure subset according to Lemma 3.4
\begin{align}
&\sup\limits_{\xi \in \hat{G}} \frac{1}{m_G(F_H)^2} \int_{F_H} \int_{F_H} \left| \frac{1}{m_G(F_n)} \int_{F_n} f(h_1h_2^{-1}gx) \xi(h_1h_2^{-1}g) \overline{f(gx)}  \overline{\xi(g)} dm_G(g) \right| dm_G(h_1) dm_G(h_2)\nonumber \\
&\quad \quad \overset{ G \text{ LCA}}=  \frac{1}{m_G(F_H)^2} \int_{F_H} \int_{F_H} \left| \frac{1}{m_G(F_n)} \int_{F_n} f(gh_1h_2^{-1}x) \overline{f(gx)} dm_G(g) \right| dm_G(h_1) dm_G(h_2) \nonumber \\
&\quad \quad \overset{n \rightarrow \infty}\longrightarrow \frac{1}{m_G(F_H)^2} \int_{F_H} \int_{F_H} \left| \int_X f(h_1h_2^{-1}y) \overline{f(y)} d\mu(y) \right| dm_G(h_1) dm_G(h_2) \\
&\quad \quad =  \frac{1}{m_G(F_H)^2} \int_{F_H} \int_{F_H} \left| \langle T_{h_2} T_{h_1^{-1}} f, f \rangle \right| dm_G(h_1) dm_G(h_2) \nonumber \\
&\quad \quad = \frac{1}{m_G(F_H^{-1})m_G(F_H)} \int_{F_H} \int_{F_H^{-1}} \left| \langle f, T_{h_1} T_{h_2} f \rangle \right| dm_G(h_1) dm_G(h_2) \nonumber.
\end{align}
In the last step we used the fact that $G$ is unimodular. This has two consequences. The first one is that $m_G(A)=m_G(A^{-1})$ for every measurable subset $A \subset G$ and the second one is that if $(  F_n )$ is a F{\o}lner-sequence then $(F_n^{-1})$ is a F{\o}lner-sequence as well, see Remark 2.2(iii).\\ \\
The above yields that for an arbitrary $\varepsilon>0$ there exists $N_1 \in \mathbb{N}$ such that for every $n \ge N_1$ 
\begin{align*}
&\sup\limits_{\xi \in \hat{G}} \frac{1}{m_G(F_H)^2} \int_{F_H} \int_{F_H} \left| \frac{1}{m_G(F_n)} \int_{F_n} f(h_1h_2^{-1}gx) \xi(h_1h_2^{-1}g) \overline{f(gx)}  \overline{\xi(g)} dm_G(g) \right| dm_G(h_1) dm_G(h_2)\\
&\quad \quad \le \frac{1}{m_G(F_H^{-1})m_G(F_H)} \int_{F_H} \int_{F_H^{-1}} \left| \langle f, T_{h_1} T_{h_2} f \rangle \right| dm_G(h_1) dm_G(h_2) + \frac{\varepsilon}{4}.
\end{align*}
Since $f \in H_{\text{wmix}}$, Lemma 2.5 guarantees that there exists $H_1 \in \mathbb{N}$, such that for all $H \ge H_1$
\begin{align*}
\frac{1}{m_G(F_H^{-1})m_G(F_H)} \int_{F_H} \int_{F_H^{-1}} \left| \langle f, T_{h_1} T_{h_2} f \rangle \right| dm_G(h_1) dm_G(h_2) \le \frac{\varepsilon}{4}.
\end{align*}
Fix this $H_1 \in \mathbb{N}$. Since $(F_n)$ is a F{\o}lner-sequence, we find $N_2 \in \mathbb{N}$, such that for all $n \ge N_2$
\begin{align*}
\sup\limits_{h \in F_{H_1}} \frac{m_G(hF_n \Delta F_n)}{m_G(F_n)} \le \min \left\{ \frac{\varepsilon}{12 \|f\|_\infty^2}, \frac{\sqrt{\varepsilon}}{2 \|f\|_\infty} \right\}
\end{align*}
by Remark 2.2(i).\\ \\
Now set $N:= \max \left\{ N_1, N_2 \right\}$. Then for all $n \ge N$ we get by Lemma 4.1
\begin{align*}
&\sup\limits_{\xi \in \hat{G}} \left| \frac{1}{m_G(F_n)} \int_{F_n} f(gx) \xi(g) dm_G(g) \right|^2\\
&\quad \le \sup\limits_{\xi \in \hat{G}}\frac{1}{m_G(F_{H_1})^2} \int_{F_{H_1}} \int_{F_{H_1}} \left| \frac{1}{m_G(F_n)} \int_{F_n} f(h_1h_2^{-1}gx) \xi(h_1h_2^{-1}g) \overline{f(gx)}  \overline{\xi(g)} dm_G(g) \right| dm_G(h_1)dm_G(h_2)\\
&\quad \quad \quad \quad \quad \quad \quad \quad \quad \quad \quad \quad \quad \quad \quad \quad + 3\sup\limits_{h \in F_{H_1}} \frac{m_G(hF_n \Delta F_n)}{m_G(F_n)}  \|f\|_\infty^2  + \left( \sup\limits_{h \in F_{H_1}} \frac{m_G(hF_n \Delta F_n)}{m_G(F_n)} \right)^2 \|f\|_\infty^2\\
&\quad\le \frac{1}{m_G(F_{H_1})m_G(F_{H_1}^{-1})} \int_{F_{H_1}} \int_{F_{H_1}^{-1}} \left| \langle f, T_{h_1} T_{h_2} f \rangle \right| dm_G(h_1) dm_G(h_2) + \frac{\varepsilon}{4}\\
&\quad \quad \quad \quad \quad \quad \quad \quad \quad \quad \quad \quad \quad \quad \quad \quad + 3 \cdot \frac{\varepsilon}{12 \|f\|_\infty^2} \|f\|_\infty^2 + \frac{\varepsilon}{4 \|f\|_\infty^2} \|f\|_\infty^2\\
&\quad \le \frac{\varepsilon}{4} + \frac{\varepsilon}{4} + \frac{\varepsilon}{4} + \frac{\varepsilon}{4} =\varepsilon
\end{align*}
and therefore we  have
\begin{align*}
\sup\limits_{\xi \in \hat{G}} \left| \frac{1}{m_G(F_n)} \int_{F_n} f(gx) \xi(g) dm_G(g) \right| \overset{N \rightarrow \infty}\longrightarrow 0,
\end{align*}
which proves the theorem for all $f \in H_\text{wmix} \cap L_\infty(X,\mu)$.\\ \\
Now suppose that $f$ is not bounded. By the density of $L_\infty$ in $L_2$ and because $H_\text{kr}=H_\text{wmix}^\perp$ comes from a factor (see argumentation in the proof of Theorem 3.1) we can find  $f_k \in H_\text{wmix} \cap L_\infty(X,\mu)$, such that $\|f-f_k\|_{2} \overset{k \rightarrow \infty}\longrightarrow 0$. Now we have
\begin{align*}
&\left| \frac{1}{m_G(F_n)} \int_{F_n} f(gx) \xi(g) dm_G(g) \right|\\
&\quad \quad \le \frac{1}{m_G(F_n)} \int_{F_n} |f(gx) - f_k(gx)| dm_G(g) + \left| \frac{1}{m_G(F_n)} \int_{F_n} f_k(gx) \xi(g) dm_G(g) \right|.
\end{align*}
By Birkhoff's ergodic theorem for amenable groups we obtain for almost every $x \in X$
\begin{align*}
\frac{1}{m_G(F_n)} \int_{F_n} |f(gx) - f_k(gx)| dm_G(g) \overset{n \rightarrow \infty}\longrightarrow \int_X |f(x)-f_k(x)| d\mu(x) = \|f-f_k\|_{1}.
\end{align*}
Now fix a $k_0 \in \mathbb{N}$ such that $\|f-f_{k_0}\|_{1} \le \|f-f_{k_0}\|_{2} < \frac{\varepsilon}{3}$. Then there exists a $N \in \mathbb{N}$, such that for all $n \ge N$ we have
\begin{align*}
\sup\limits_{\xi \in \hat{G}} \left| \frac{1}{m_G(F_n)} \int_{F_n} f_{k_0}(gx) \xi(g) dm_G(g) \right| \le \frac{\varepsilon}{3}
\end{align*}
and
\begin{align*}
\frac{1}{m_G(F_n)} \int_{F_n} |f(gx) - f_{k_0}(gx)| dm_G(g) \le \|f-f_{k_0}\|_{1} + \frac{\varepsilon}{3}.
\end{align*}
This gives us for all $n \ge N$
\begin{align*}
\sup\limits_{\xi \in \hat{G}} \left| \frac{1}{m_G(F_n)} \int_{F_n} f(gx) \xi(g) dm_G(g) \right| \le \varepsilon
\end{align*}
and the first claim is completely proven.
\newpage
\noindent Now consider the uniquely ergodic case. Then we have uniform convergence of the ergodic averages for (arbitrary) F{\o}lner-sequences in the Birkhoff ergodic theorem, i.e.
\begin{align*}
\frac{1}{m_G(F_n)} \int_{F_n} f(gx) dm_G(g) \overset{n \rightarrow \infty} \longrightarrow \int_X f(y) d\mu(y)
\end{align*} 
uniformly in $x \in X$. This follows by a simple modification of \cite[Theorem 2.8]{assani} again by replacing Ces\`{a}ro-averages by F{\o}lner-averages. Now the convergence in (4.4) is also uniform and we are able to continue the proof in the same manner as before. We still need the strong F{\o}lner-property in order to apply Lemma 3.4 in this proof. Note that every continuous function on a compact space is bounded and therefore we do not need the approximation argument at the end of the proof in this case.
\end{proof}
\begin{rem}
\begin{enumerate}
\item[(i)] Zorin-Kranich (see \cite{pavel}) showed that the claim in Theorem 3.1 remains true, if the group $G$ is an arbitrary amenable group and the action is only measurable and not necessarily continuous. Thus the question arises whether the claim in Theorem 1.1 remains true as well if we consider arbitrary measurable actions of amenable groups. Note that only the proof of Lemma 3.4 relies on the continuity of the group action and an important tool in the proof of the main theorem, Lemma 4.1, is true for general amenable groups. 
\item[(ii)] Note that we do not need to apply Lemma 3.4 in the proof of Theorem 3.1 and Theorem 1.1, if the group $G$ is countable. By going through the proof carefully, it turns out that both theorems remain true if we consider arbitrary measurable actions of countable LCA-groups.
\end{enumerate}
\end{rem}
\begin{bsp}
Consider $G=(\mathbb{Z}^d,+)$. It is well known that
\begin{align*}
F_n= \left\{ -n, -(n-1), \ldots , n-1,n \right\}^d
\end{align*}
is a tempered strong F{\o}lner-sequence in $G$ and the counting measure is a Haar-measure on $G$. Every character $\xi: G \longrightarrow \mathbb{T}$ is of the form
\begin{align*}
\xi(n_1, \ldots , n_d)= \lambda_1^{n_1} \cdot \ldots \cdot \lambda_d^{n_d}
\end{align*}
for some $\lambda_1, \ldots , \lambda_d \in \mathbb{T}$.\\ \\
By Theorem 3.1 we obtain that for every ergodic action of $G$ on $X$, where $(X,\mu)$ is a probability space, and every $f \in L_1(X)$, there exists a full measure subset $X' \subset X$ such that 
\begin{align*}
\frac{1}{(2N+1)^d} \sum\limits_{n_1, \ldots , n_d=-N}^N \lambda_1^{n_1} \cdot \ldots \cdot \lambda_d^{n_d} f((n_1, \ldots , n_d)x)
\end{align*}
converges for all $\lambda_1, \ldots , \lambda_d \in \mathbb{T}$ and every $x \in X'$.\\ \\
Now consider a Bernoulli-shift, i.e. let $G$ act on $X = \left\{ 0, \ldots , k-1\right\}^{\mathbb{Z}^d}$ by shift, where $k \ge 2$ is an integer and $X$ is equipped with the product $\sigma$-algebra. It is well known that this system is strongly mixing, which has the consequence that the Kronecker factor consists of constant functions only. Therefore by Theorem 1.1 we obtain that for every $f \in L_2(X)$ with $\langle f, 1 \rangle =0$ there exists a subset $X' \subset X$ with $\mu(X')=1$, such that 
\begin{align*}
\sup\limits_{\lambda_1, \ldots , \lambda_d \in \mathbb{T}} \left| \frac{1}{(2N+1)^d} \sum\limits_{n_1, \ldots ,n_d=-N}^N \lambda_1^{n_1} \cdot \ldots \cdot \lambda_d^{n_d} f((x_{j+n_1}^{(1)}, \ldots , x_{j+n_d}^{(d)})_{j \in \mathbb{Z}}) \right| \overset{N \rightarrow \infty}\longrightarrow 0
\end{align*}
for every $x=(x_j^{(1)}, \ldots , x_j^{(d)})_{j \in \mathbb{Z}} \in X'$.
\end{bsp}
\noindent \textbf{Acknowledgement.} The author is grateful for the support of the International Max Planck Research School for Mathematics in the Sciences Leipzig. Moreover he wants to thank Tanja Eisner and Felix Pogorzelski for several helpful discussions.


\begin{thebibliography}{9}
\bibitem{assani} I. Assani, Wiener Wintner Ergodic Theorems, World Scientific Publishing, 2003.
\bibitem{bekka} M. Bachir Bekka, M. Mayer, Ergodic Theory and Topological Dynamics of Group Actions on Homogeneous Spaces, Cambridge University Press, 2000. 
\bibitem{bartos} W. Bartoszek, A. {\'S}piewak, \textit{A note on a Wiener-Wintner theorem for mean ergodic Markov amenable semigroups}, Proc. Amer. Math. Soc., 145 (2017), no. 7, 2997–3003
\bibitem{bellow} A. Bellow and V. Losert, \textit{The weighted pointwise ergodic theorem and the individual ergodic theorem along subsequences}, Trans. Amer. Math. Soc., 288 (1985), 307-345.
\bibitem{duven2} C. Beyers, R. Duvenhagem, A. Stroh, \textit{A van der Corput lemma and weak mixing over groups}, arXiv:math/0512059v1.
\bibitem{bourgain} J. Bourgain, \textit{Double recurrence and almost sure convergence}, J. Reine Angew. Math. 404 (1990), 140-161.
\bibitem{duven} R. Duvenhage, \textit{Bergelson's Theorem for weakly mixing C*-dynamical systems}, Studia Math. 192 (2009), 235-257.
\bibitem{eisner} T. Eisner, B. Farkas, M. Haase and R. Nagel, Operator Theoretic Aspects of Ergodic Theory, Springer Verlag, 2015.
\bibitem{eiskra} T. Eisner, P. Zorin-Kranich, \textit{Uniformity in the Wiener-Wintner theorem for nilsequences}, Discrete Contin. Dyn. Syst. 33 (2013), 3497-3516.
\bibitem{emerson} W. R. Emerson, \textit{Ratio properties in locally compact amenable groups}, Trans. Amer. Math. Soc. 133 (1968), no. 1, 179-204.
\bibitem{fan} A. Fan, \textit{Topological Wiener-Wintner ergodic theorem with polynomial weights}, Chaos, Solitons \& Fractals, vol. 117, 105-116 (2018).
\bibitem{franz} N. Frantzikinakis, \textit{Uniformity in the polynomial Wiener-Wintner theorem}, Ergodic Theory Dynam. Systems 26 (2006), no. 4, 1061-1071.
\bibitem{furstenberg} H. Furstenberg, Stationary Processes and Prediction Theory, Princeton University Press, 1960.
\bibitem{green} F. P. Greenleaf, Invariant Means on Topological Groups, Van Nostrand Reinhold Companies New York, 1969.
\bibitem{host} B. Host and B. Kra, \textit{Uniformity seminorm on $l^\infty$ and applications}, J. Anal. Math. 108 (2009), 219-276.
\bibitem{kuipers} L. Kuipers, H. Niederreiter, Uniform Distribution of Sequences, John Wiley and Sons, 1974.
\bibitem{lenz} D. Lenz, \textit{Continuity of eigenfunctions of uniquely ergodic dynamical systems and intensity of Bragg Peaks}, Math. Phys. 287 (2009), 225-258.
\bibitem{lesign2} E. Lesigne, \textit{Spectre quasi-discret et th\'{e}or\`{e}me ergodique de Wiener-Wintner pour les polyn\^{o}mes}, Ergodic Theory Dynam. Systems 13 (1993), no. 4, 767-784.
\bibitem{lesign1} E. Lesigne, \textit{Un th\'{e}or\`{e}me de disjunction de syst\`{e}mes dynamiques et une g\'{e}n\'{e}ralisation du th\'{e}or\`{e}me ergodique de Wiener-Wintner}, Ergodic Theory Dynam. System 10 (1990), no. 3, 513-521.
\bibitem{linden} E. Lindenstrauss, \textit{Pointwise theorems for amenable groups}, Invent. Math., vol. 146 (2001), 259-295.
\bibitem{pier} J.-P. Pier, Amenable locally compact groups, John Wiley and Sons, 1984.
\bibitem{pog} F. Pogorzelski: \textit{Banach space-valued ergodic theorems and spectral approximation}, Dissertationsschrift Friedrich-Schiller-Universität Jena, 2014.
\bibitem{schwarzenberger} F. Pogorzelski, F. Schwarzenberger: \textit{A Banach space-valued ergodic theorem for amenable groups and applications}, J. d'Anal. Math. (1) 130 (2016), 19-69.
\bibitem{robinson} E. A. Robinson, \textit{On uniform convergence in the Wiener-Wintner theorem}, J. London Math. Soc. (2) 49 (1994), no. 3, 493-501.
\bibitem{runde} V. Runde, Lectures on Amenability, Springer Verlag, 2001.
\bibitem{schreiber} M. Schreiber, \textit{Topological Wiener-Wintner theorems for amenable operator semigroups}, Ergodic Theory Dynam. Systems 34 (2014), no. 5, 1674-1698.
\bibitem{ww} N. Wiener, A. Wintner, \textit{Harmonic analysis and ergodic theory}, Amer. J. Math 63 (1941), 415-426.
 \bibitem{pavel} P. Zorin-Kranich, \textit{Return times for amenable groups}, Israel J. Math. 204 (2014), 85-96.
\end{thebibliography}
\end{document}